\documentclass[12pt,a4paper]{amsart}

\usepackage{amsfonts, amsmath, amssymb, amsthm, amscd, hyperref}

\newtheorem{thm}{Theorem}[section]
\newtheorem*{thm*}{Theorem}
\newtheorem{lem}[thm]{Lemma}

\newtheorem{cor}[thm]{Corollary}

\theoremstyle{definition}
\newtheorem{defn}[thm]{Definition}

\theoremstyle{remark}
\newtheorem{rem}[thm]{Remark}
\newtheorem{exa}[thm]{Example}

\numberwithin{equation}{section}

\newcommand{\Deltad}{{\Delta'}}
\newcommand{\Fp}{\mathbb F_p}
\newcommand{\Zp}{\mathbb Z/p}
\newcommand{\Ztwo}{\mathbb Z/2}

\DeclareMathOperator{\hind}{ind}

\DeclareMathOperator{\gfree}{g_{\rm free}}

\DeclareMathOperator{\Map}{Map}

\renewcommand{\int}{\mathop{\rm int}}
\renewcommand{\epsilon}{\varepsilon}

\begin{document}


\title[Configuration-like spaces and coincidences\dots]{Configuration-like spaces and coincidences of maps on orbits}

\author{R.N.~Karasev}
\address{
Roman Karasev, Dept. of Mathematics, Moscow Institute of Physics
and Technology, Institutskiy per. 9, Dolgoprudny, Russia 141700}
\email{r\_n\_karasev@mail.ru}
\urladdr{http://www.rkarasev.ru/en/}
\thanks{The research of R.N.~Karasev is supported by the Dynasty Foundation, the President's of Russian Federation grant MK-113.2010.1, the Russian Foundation for Basic Research grants 10-01-00096 and 10-01-00139, the Federal Program ``Scientific and scientific-pedagogical staff of innovative Russia'' 2009--2013}

\author{A.Yu.~Volovikov}
\address{
Alexey Volovikov, Department of Mathematics, University of Texas at Brownsville, 80
Fort Brown, Brownsville, TX, 78520, USA}
\address{
Alexey Volovikov, Department of Higher Mathematics, Moscow State Institute of
Radio-Engineering, Electronics and Automation (Technical
University), Pr. Vernadskogo 78, Moscow 117454, Russia}
\email{a\_volov@list.ru}
\thanks{The research of A.Yu.~Volovikov was partially supported by the Russian Foundation for Basic Research grant No. 08-01-00663 and by grant of Program of Support of Leading Scientific Schools HI\!I\!I-1562.2008.1}





\begin{abstract}
In this paper we study the spaces of $q$-tuples of points in a Euclidean space, without $k$-wise coincidences (configuration-like spaces). A transitive group action by permuting these points is considered, and some new upper bounds on the genus (in the sense of Krasnosel'skii--Schwarz and Clapp--Puppe) for this action are given. Some theorems of Cohen--Lusk type for coincidence points of continuous maps to Euclidean spaces are deduced.
\end{abstract}

\maketitle

\section{Introduction}

In this paper we address the question of finding some sufficient conditions that guarantee that a continuous map $f : X\to \mathbb R^m$ has a certain number of self-coincidences on an orbit of a $G$-action on $X$, where $G$ is a finite group. The most famous result of this kind is the Borsuk--Ulam theorem~\cite{bor1933}, where $X$ is the $m$-dimensional sphere and $G=\Ztwo$ acts on $X$ by the antipodal action. Partial solutions of the Knaster problem in~\cite{mak1989,vol1992} provide another application.

In order to simplify the statements we need some definitions.

\begin{defn}
Let $G$ be a finite group, and $X$ be a $G$-space, i.e. a topological space with continuous left $G$-action. For a given continuous map $f :X\to Y$ we denote by $A(f, k)$ the coincidence set
\begin{multline*}
A(f, k) = \{x\in X : \exists\ \text{distinct}\ g_1,g_2,\ldots,g_k\in G\\ \text{such that}\ f(g_1x) = f(g_2x) = \dots = f(g_kx)\}.
\end{multline*}
\end{defn}

Generally, to deduce existence theorems for coincidences, we have to define the complexity of the action of $G$ on the space $X$. The following definition was made for $G=\Ztwo$ by Krasnosel'skii and Yang in~\cite{kr1952,kr1955A,yang1955,yang1955II}, and for arbitrary finite $G$ by Krasnosel'skii in~\cite{kr1955B}, as noted in~\cite{schw1957}. It is usually called the Krasnosel'skii--Schwarz genus.

\begin{defn}
\emph{The free genus} of a free $G$-space $X$ is the least number $n$ such that $X$ can be covered by $n$ open subsets $X_1,\ldots, X_n$ so that for every $i$ there exists a $G$-equivariant map $X_i\to G$. We denote the free genus by $\gfree(X)$.
\end{defn}

In this definition it makes sense to consider paracompact spaces $X$ only, see Section~\ref{free-genus} for more details. In this case, we can take closed sets in the definition, instead of open sets. Thus, in the sequel we consider paracompact $G$-spaces, unless otherwise stated. The following theorem was proved in~\cite{schw1966}. In this theorem and in the rest of the paper $p$ is a prime number, $\Zp$ is the cyclic group of order $p$, and $\Fp$ is the same group, considered as a field.

\begin{thm}
\label{coinc-schw}
Let $X$ be a free connected $\Zp$-space. Assume that $\gfree(X)>m(p-1)$. Then for any continuous map $f : X\to \mathbb R^m$
$$
\gfree (A(f,p))\ge \gfree(X)-m(p-1).
$$
In particular, the set $A(f, p)$ is non-empty. 
\end{thm}

In~\cite{colu1976} the partial coincidences on an orbit were considered, and the following theorem was proved.

\begin{thm}
\label{coinc-colu}
Let $X$ be a free $\Zp$-space. Let $(p+1)/2\le k \le p$ or $k=2$. Suppose that $X$ is connected and acyclic over the field $\Fp$ in dimensions less than $(m-1)(p-1)+k-1$. Then for any continuous map $f : X\to \mathbb R^m$ we have $A(f,k)\not=\emptyset$.
\end{thm}

It was conjectured in~\cite{colu1976} that the restrictions on $k$ are not necessary. A step in this direction was made in~\cite{bol2001}, where the case $k=\dfrac{p-1}{2}$ was considered, in~\cite{vol2005,vol2007} this conjecture was proved for $k\neq 3$. Here we assume $p$ to be odd, because the case $p=2$ is already covered by Theorem~\ref{coinc-colu}.

In this paper we establish upper and lower bounds on the genus $\gfree$ of certain configuration-like spaces in Section~\ref{conf-genus}. Among the consequences is the following result that incorporates both cited theorems without any restrictions on $k$ in Theorem~\ref{coinc-colu}.

\begin{cor}
\label{coinc-zp}
Let $X$ be a free $G$-space, where $G$ is a cyclic group of order $p$. Assume that $\gfree(X)>(m-1)(p-1)+k-1$. Then for any continuous map $f : X\to \mathbb R^m$
$$
\gfree(A(f,k))\ge \gfree(X)-(m-1)(p-1)-k+1.
$$
\end{cor}

Note that the properties of $\gfree$ in Section~\ref{free-genus} imply that 
$$
\dim A(f,k)\ge \gfree(X)-(m-1)(p-1)-k,
$$
and therefore Corollary~\ref{coinc-zp} confirms the conjecture from~\cite{colu1976}.

In fact, Corollary~\ref{coinc-zp} follows from a more general statement; to formulate it we need some more definitions.

\begin{defn}
Let $G$ be a finite group, $X$ be a $G$-space. For a given continuous function $f :X\to \mathbb R$ denote by $A'(f, k)$ the \emph{maximum coincidence} set
\begin{multline*}
A'(f, k) = \{x\in X : \exists\ \text{distinct}\ g_1,g_2,\ldots,g_k\in G\\\text{such that}\ f(g_1x) = f(g_2x) = \dots = f(g_kx) = c\\ \text{and}\ \forall g\in G\ f(gx)\le c\}.
\end{multline*}
It is clear that $A'(f, k)\subseteq A(f, k)$.
\end{defn} 

\begin{defn}
If the group $G$ acts on $X$ without $G$-fixed points, we call $X$ \emph{fixed point free}.
\end{defn}

For fixed point free $G$-spaces a notion similar to the Krasnosel'skii--Schwarz genus $\gfree$ can be defined; see Section~\ref{fpf-genus} for the definition of $g_G(X)$ and its properties. For a cyclic group of prime order these two notions coincide. From the upper bounds on the genus of certain configuration-like spaces in Section~\ref{conf-genus}, we deduce the following result on coincidences.

\begin{thm}
\label{coinc-fpf}
Let $G$ be a finite group, $m\ge 1$, $2\le k\le |G|$ be integers. Consider a fixed point free $G$-space $X$ with $g_G(X) > (|G|-1)(m-1) + k - 1$ and two continuous maps $f_1 : X\to \mathbb R$ and $h : X\to \mathbb R^{m-1}$. Then 
\begin{multline*}
g_G(A(f_1\oplus h, k))\ge g_G(A'(f_1,k)\cap A(h,q)) \ge\\ \ge g_G(X) - (|G|-1)(m-1) - k + 1.
\end{multline*}
\end{thm}

Here we denote by $f_1\oplus h$ the map to $\mathbb R\oplus \mathbb R^{m-1}$ with components $(f_1, h)$. This result works well at least for $p$-tori (groups of the form $(\Zp)^k$), see Sections~\ref{fpf-genus} and \ref{manifolds} for details. Since any nontrivial finite group has a subgroup of this kind (e.g. a cyclic subgroup of prime order), we can sometimes replace $G$ by an appropriate $p$-torus subgroup (if this subgroup acts without fixed points on $X$).

In the case when $X$ is a manifold, the following lower bound for $\dim A(f,k)$ was found in~\cite{colu1976}: under the conditions of Theorem~\ref{coinc-colu}, if $X$ is an $\Fp$-orientable connected $N$-dimensional manifold, acyclic over the field $\Fp$ in dimensions less than $(m-1)(p-1)+k-1$, then $\dim A(f,k)\ge N-(m-1)(p-1)-k+1$. In Section~\ref{manifolds} we prove a similar result for $p$-tori.

Another corollary is a Knaster-type result, similar to results of~\cite{karvol2009}, see also~\cite{kna1947} for the formulation of Knaster's conjecture.

\begin{defn}
Denote by $I[G]\subset\mathbb R[G]$ the $G$-invariant subspace in the group ring $\mathbb R[G]$ consisting of
$$
\sum_{g\in G} \alpha_g g,\quad\text{with}\quad \sum_{g\in G} \alpha_g = 0.
$$
\end{defn}

\begin{cor}
\label{knaster1}
Consider a $p$-torus $G=(\Zp)^n$ for odd $p$ and set $q=|G|$. Let $S^{q-2}$ be the unit sphere of $I[G]$ with respect to some $G$-invariant inner product. Suppose $f : S^{q-2}\to \mathbb R$ is a continuous function, and $x\in S^{q-2}$ is some point. Then there exists a rotation $\rho$ of $S^{q-2}$ with positive determinant, such that
$$
\forall g\in G\setminus\{e\}\ f(\rho (gx)) = c,\ \text{and}\ f(\rho(x)) \le c.
$$
\end{cor}

To prove this corollary we use another numerical invariant of $G$-action, defined in~\cite{vol2000}. Here we give the definition for connected spaces only.

\begin{defn}
Let $G=(\Zp)^n$ be a $p$-torus, and $X$ be a connected $G$-space. Consider the Leray--Serre spectral sequence with 
$$
E_2^{*,*} = H^*(BG, H^*(X, \Fp)),
$$
converging to the equivariant cohomology of $X$ in the sense of Borel $H_G^*(X, \Fp)$. Define $i_G(X)$ to be the minimum $r\ge 2$ such that the image of $d_r$ in the bottom row $E_r^{*,0}$ is nonzero.
\end{defn}

\begin{proof}[Proof of Corollary~\ref{knaster1}]
This corollary follows from Theorem~\ref{coinc-fpf}. Consider the space $SO(q-1)$ along with the action of $G$ by right multiplication by $g^{-1}$ ($g\in G$), this is a left $G$-action. The function $f$ induces the function on $SO(q-1)$ by the formula
$$
\tilde f : \rho \mapsto f(\rho x).
$$ 
We have to prove that $A'(\tilde f, q-1)\neq \emptyset$. By Theorem~\ref{coinc-fpf} (case $m=1$) it suffices to show that $g_G(SO(q-1))\ge q-1$ for the this action of $G$ on $SO(q-1)$. 

It was shown in~\cite[Proposition~4.7]{vol2005} that $g_G(X)\ge i_G(X)$ for fixed point free $G$-spaces $X$, and the value $i_G(SO(q-1)) = q-1$ was found in~\cite{karvol2009}.
\end{proof}

In Section~\ref{conf-genus} we shall prove a Corollary~\ref{lower-bound-pconf-manifold} that estimates the genus of the classical configuration space of an $\Fp$-oriented $m$-dimensional manifold $M$ from below by the number $(m-1)(p-1) + 2$. Here $p$ is a prime, the configuration space consists of $p$-tuples of distinct points in $M$, and the genus is taken with respect to the cyclic permutation action of $\Zp$. Note that for such groups the free genus and the fixed point free genus coincide. This result is also valid for the free genus with respect to the action of the full permutation group $\Sigma_p$, because this group acts freely on the configuration space (see Property~\ref{free-subgr} of the free genus in Section~\ref{free-genus}). In~\cite[Theorem~5.2]{bgrt2010} it is shown that for the case $M=S^m$ this bound is optimal, because the space of configurations of $n$ distinct points in $S^m$ ($n$ does not need to be a prime) is $\Sigma_n$-equivariantly homotopy equivalent to a polyhedron of dimension $(m-1)(n-1)+1$. This means that the bound for $\Zp$-actions is also optimal for spheres.

The rest of the paper is organized as follows. In Sections~\ref{free-genus} and \ref{fpf-genus} we give the definitions and properties of the free genus and the fixed point free genus. In Section~\ref{conf-spaces} we define different configuration-like spaces. In Section~\ref{conf-genus} we give lower and upper bounds for the genus of configuration-like spaces; these are the core results of the paper. In Section~\ref{coinc-proof-sec} we deduce the coincidence theorems. In Section~\ref{manifolds} we improve the coincidence theorems in case the domain space is a manifold.

In this paper we generally use purely geometric methods, based on the subadditivity, the dimension upper bound, and other properties of the genus. The reader may compare this approach with the lower bounds for the genus (actually for the number $i_G$) in~\cite{kar2009}, made with computations in cohomology and spectral sequences.

The authors thank Peter~Landweber for numerous remarks and corrections, and Jes\'us Gonz\'alez for pointing out the upper bounds for the homotopy dimension of the configuration spaces of a sphere.

\section{Genus of a free action}
\label{free-genus}

It is well known (e.g. see~\cite{bart1993}) that for paracompact spaces the definition of $\gfree$ can be reformulated as follows, using a partition of unity argument.

\begin{defn}
\emph{The free genus} of a free $G$-space $X$ is the least number $n$ such that $X$ can be $G$-mapped to some $n$-fold join $G*\dots *G$.
\end{defn}

Here we list the properties of the free genus, mainly from~\cite{schw1957}. We use the notation $\gfree(X, G)$, when it is needed to indicate explicitly the acting group.

\begin{enumerate}
\item (Monotonicity)
If there is a $G$-map $f: X\to Y$, then $ \gfree(X) \le \gfree(Y)$;

\item (Subadditivity)
Let $X=A\cup B$, where $A$, $B$ are closed or open $G$-invariant subspaces. Then $\gfree(X) \le \gfree(A)+\gfree(B)$;

\item (Dimension upper bound)
$\gfree(X)\le \dim X+1$;

\item (Cohomology lower bound)
Assume that the order of $G$ is divisible by a prime $p$ and $X$ is connected and acyclic over $\Fp$ in degrees $\le N-1$, i.e. $H^i(X, \Fp)=0$ for all $0<i<N$. Then $\gfree(X) \ge N+1$. In particular, if $X$ is connected and acyclic over $\Fp$ in degrees $\le N-1$ and $G$ is a $p$-torus then $\gfree (X) \ge N+1$.

\item (Passing to a subgroup)\
\label{free-subgr}
If $X$ is a free $G$-space, and $F$ is a nontrivial subgroup of $G$, then $\gfree(X, G) \ge \gfree(X, F)$.
\end{enumerate}

Note that in the cohomology lower bound it is convenient to use \v Cech or Alexander--Spanier cohomology because they satisfy the continuity property. That is, if a cohomology class $\xi\in H^*(X)$ has nonzero restriction to a closed subspace $Y$, then it has nonzero restriction to an open subset $U\supseteq Y$. This property is useful to handle ``pathological'' subspaces $Y$. The cohomology lower bound in fact has the more general form
$$
\gfree(X)\ge \hind(X)+1,
$$
where $\hind(X)$ is the maximum dimension $n$ such that the natural map $H^n(BG, M)\to H_G^n(X, M)$ is non-trivial for some $G$-module $M$ (compare with the definition of $i_G$). This estimate for the free genus was established in~\cite{yang1955,cf1960} for $G=\Ztwo$ and in~\cite{schw1966} for arbitrary $G$. In~\cite{schw1966} the value $\hind(X)+1$ was called the \emph{homological genus}.

\section{Genus of a fixed point free space}
\label{fpf-genus}

The definition of the free genus does not allow one to work with $G$-spaces that are not free. Different definitions of the genus for non-free actions of $G$ were given in~\cite{clp1986,bart1990,clp1991}. The most general definitions can be found in the book~\cite{bart1993}; the case of compact Lie group instead of a finite group is also considered there. Here we make use of a certain kind of genus of actions of a finite group $G$ without fixed points, i.e. actions on a space $X$ such that the stabilizer of any point $x\in X$ is a proper subgroup of $G$.

\begin{defn}
Denote by $D_G$ the disjoint union of all orbits $G/H$, where $H$ is a subgroup of $G$ not equal to $G$.
\end{defn}

Note that any discrete fixed point free $G$-space can be $G$-mapped to $D_G$. Similar to the case of the free genus, the definition (following~\cite{bart1993}) of the fixed point free genus can be given in two different ways, that are equivalent for paracompact $G$-spaces through a standard partition of unity argument.

\begin{defn}
\emph{The fixed point free genus} of a fixed point free $G$-space $X$ is the least number $n$ such that $X$ can be $G$-equivariantly mapped to the $n$-fold join $D_G*\dots *D_G$. We denote the fixed point free genus by $g_G(X)$.
\end{defn}

\begin{defn}
\emph{The fixed point free genus} of a fixed point free $G$-space $X$ is the least number $n$ such that $X$ can be covered by $n$ open subsets $X_1,\ldots, X_n$ so that every $X_i$ can be $G$-mapped to $D_G$.
\end{defn}

Again, we state some properties of the fixed point free genus (see~\cite[Section~4.4]{vol2005} for proofs).

\begin{enumerate}
\item (Monotonicity)
If there is a $G$-map $f: X\to Y$, then $g_G(X) \le g_G(Y)$;

\item (Subadditivity)
Let $X=A\cup B$, where $A$, $B$ are closed or open $G$-invariant subspaces. Then $g_G(X) \le g_G(A)+g_G(B)$;

\item (Dimension upper bound)
$g_G(X)\le \dim X+1$;

\item (Cohomology lower bound)
If $X$ is connected and acyclic over $\Fp$ in degrees $\le N-1$ and $G$ is a $p$-torus then $g_G(X) \ge N+1$.
\end{enumerate}

It is clear from the definition, that $\gfree$ and $g_G$ coincide for groups $\Zp$, because for such groups $G=D_G$. 

Let us give an example that shows the difference between the free genus and the fixed point free genus. 

\begin{exa}
Let $G=\mathbb Z/r$, where $r=ab$ is a product of two coprime integers $a,b > 1$ and let $EG$ be the universal free $G$-space. Then $\gfree(EG)$ is infinite (it follows from the cohomology lower bound), while $g_G(EG)$ is finite. The second claim follows easily from the example of Conner and Floyd~\cite{cf1959} of a contractible finite simplicial complex $C$, on which $G$ acts without fixed points. Then $EG\times C$ is contractible and the diagonal action of $G$ on this space is free, so $EG\times C$ is $G$-homotopy equivalent to $EG$ by a simple equivariant obstruction theory argument (equivariant obstruction theory coincides with ordinary obstruction theory when the action is free). From the natural projection $EG\times C \to C$, the monotonicity, and the dimension upper bound on $g_G$ we obtain $g_G(EG)\le g_G(C)\le \dim C+1 < +\infty$.
\end{exa}

As was already mentioned in the proof of Corollary~\ref{knaster1}, the cohomology lower bound was generalized in~\cite[Proposition~4.7]{vol2005} to give the estimate
$$
g_G(X)\ge i_G(X).
$$

\section{Definitions of configuration-like spaces}
\label{conf-spaces}

First, we need some definitions of configuration spaces for a topological space $Y$. Actually, the space $Y$ will often be $\mathbb R^m$ in this paper.

\begin{defn}
Define the \emph{$k$-wise diagonal} 
\begin{multline*}
\Delta_q^k(Y) = \{(y_1, \ldots, y_q)\in Y^q : y_{i_1} = y_{i_2} = \dots = y_{i_k}\\ \text{for some}\ 
i_1 < i_2 < \dots < i_k\}.
\end{multline*}
\end{defn}

\begin{defn}
Define the \emph{$k$-wise maximum diagonal}
\begin{multline*}
\Deltad_q^k(\mathbb R) = \{(y_1, \ldots, y_q)\in \mathbb R^q : y_{i_1} = y_{i_2} = \dots = y_{i_k} = c\\ \text{for some}\ i_1 < i_2 < \dots < i_k,\ \text{and}\ \forall j\ y_j\le c\}.
\end{multline*}
It is clear that $\Deltad_q^k(\mathbb R)\subseteq \Delta_q^k(\mathbb R)$; for $k = q$ these diagonals are equal.
\end{defn}

\begin{defn}
Denote the \emph{configuration-like spaces} 
$$
V(Y,q,k) =Y^q\setminus \Delta_q^k(Y),\quad V(m,q,k)=V(\mathbb R^m,q,k).
$$ 
Put also 
$$
W(q,k)=\mathbb R^q\setminus \Deltad_q^k(\mathbb R).
$$
\end{defn}

Note that $W(q,k)\supseteq V(1, q, k)$ in general, and $W(q,q)=V(1,q,q)$ (since $\Delta_q^q(\mathbb R)=\Deltad_q^q(\mathbb R)$). The spaces $V(Y, q, k)$ are denoted $G(Y, q, k)$ in~\cite{colu1976}, but we do not use the letter $G$ here to avoid confusion with the group action.

We also need the configuration-like spaces of the following type.

\begin{defn}
Consider $\mathbb R^m = \mathbb R\times \mathbb R^{m-1}$ and denote
$$
V_1(m,q,k) = (\mathbb R^m)^q\setminus \Delta_q^k(\mathbb R)\times\Delta_q^q(\mathbb R^{m-1}).
$$ 
\end{defn}
These spaces may seem unnatural, but their genus can be calculated precisely in all cases, see Section~\ref{conf-genus}. Note that $V(m, q, k)\subseteq V_1(m, q, k)$.

In the sequel we usually consider a finite group $G$ with $|G|=q$; in this case we identify $Y^q$ with the space of maps $\Map(G, Y)$. The latter space has the natural $G$-action given in the following definition.

\begin{defn}
For a finite group $G$ and a topological space $Y$ consider the space $\Map(G,Y)$ of maps with the usual left $G$-action, i.e. for $\phi \in\Map(G,Y)$ we define $g\phi \in\Map(G,Y)$ as follows. For any $h\in G$ put
$$
(g\phi)(h) = \phi(g^{-1}h).
$$
\end{defn}

Note that the group ring $\mathbb R[G]$ can be identified with $\Map(G,\mathbb R)$ by the assignment $\phi\leftrightarrow \sum \phi(g)g$. It can be easily checked that this identification transforms the action of $G$ on $\Map(G,\mathbb R)$ defined above to the usual left action of $G$ on its group ring $\mathbb R[G]$.

\section{Genus of configuration-like spaces}
\label{conf-genus}

The index of configuration spaces $V(2, n, 2)$ was estimated from below in the papers of Smale and Vasil'ev~\cite{sma1987,vass1988}, to give lower bounds of the ``topological'' complexity of algorithms for finding the roots of a complex polynomial. In~\cite{rot2008,kar2009} similar estimates were given for $V(m, n, 2)$. In those papers the genus was considered with respect to the free action of the permutation group $\mathfrak S_n$ by permuting the points of configuration.

For the action of a $p$-torus $G$ (with $|G|=q$) some lower bounds on the genus of $V(m, q, k)$ were obtained in~\cite{vol2005,vol2007,kar2009}, using the homological lower bound, and a more accurate bound by the homological index $i_G(V(m, q, k))$.

We start with a simple geometric upper bound, valid for any finite group $G$.

\begin{thm}
\label{max-conf-space-genus} For a finite group $G$ and $2\le k\le q = |G|$ we have
$$
g_G(W(q, k))\le k-1.
$$ 
\end{thm}

\begin{proof}
By $\binom{G}{m}$ we denote the set of all $m$-element subsets of $G$ and consider it as a discrete $G$-space. For $M=\{g_1,\dots,g_m\}\in \binom{G}{m}$ and $g\in G$ we define $gM=\{gg_1,\dots, gg_m\}$.

The space $\binom{G}{m}$ has no $G$-fixed points for $m<q$. Indeed, if $M=\{g_1,\ldots, g_k\}\in\binom{G}{[k]}$ is a fixed point and $g\in G$ is some element, then the set $M'=\{g_1,\ldots, g_k\}\cdot \left(g^{-1}g_1\right)^{-1}\ni g$, but $M=M'$ (this is a fixed point), so $g\in M$ for every $g\in G$, hence $M=G$.

We identify $W(q, k)$ with a subset of $\Map(G, \mathbb R)$, as described above. For a nonempty subset $M\subset G$ we define an open subset $U_M\subset \Map(G, \mathbb R)$ by (here $\overline M=G\setminus M$ denotes the complement of $M$)
$$
U_M = \{\phi\in \Map(G, \mathbb R) : \phi(g) > \phi(h)\ \forall\ g\in M, h\in \overline M\},
$$
and for any $1\le m\le q-1$ we put
$$
V_m = \bigcup_{M\subset G,\ |M| = m} U_M.
$$

Now we are going to prove that:

\begin{enumerate}
\item
\label{a_disj}
For any different $M\not=M'$ of same size $|M|=|M'|$ the sets $U_M$ and $U_{M'}$ are disjoint;

\item
\label{a_action}
For any $g\in G$ we have $gU_M=U_{gM}$;

\item
\label{a_inv}
$V_m$ are invariant subspaces of $\Map(G,\mathbb R)$;

\item 
\label{a_uni}
We have the equality $W(q,k) = V_1\cup V_2\cup\dots\cup V_{k-1}$ and each $W(q, k-1)$ is open in $W(q, k)$;

\item
\label{a_cont}
The map $f_m : V_m\to \binom{G}{m}$ defined by $f_m(U_M) = M$ is continuous and $G$-equivariant.
\end{enumerate}

Assertion~(\ref{a_disj}) is almost obvious. For $M\subset G$ with $|M|=k$ the set $U_M$ consists of $\phi:G\to\mathbb R$ such that the $k$ largest elements in the set $\phi(G)$ are defined and correspond to $M$.

To prove assertion~(\ref{a_action}) it is enough to show that if $\phi\in U_M$ then $g\phi\in
U_{gM}$. Consider $\psi=g\phi\in gU_M$. Since $g^{-1}\psi=\phi\in U_M$ we have $\forall\ h\in M, h'\in \overline M$ the inequality $(g^{-1}\psi)(h)>(g^{-1}\psi)(h')$, i.e. $\psi(gh)>\psi(gh')$. Since $gh\in gM$ and $gh'\in g(\overline  M)=\overline{gM}$, we see that $\psi\in U_{gM}$, note that when $h$ runs through $M$ and $h'$ through $\overline M$ elements $gh$ and $gh'$ runs through $gM$ and $\overline{gM}$ respectively.

Assertion~(\ref{a_inv}) follows from~(\ref{a_action}).

To prove assertion~(\ref{a_uni}) note that $W(q,k)$ is the set of $\phi:G\to\mathbb R$ such that the maximum of $\phi$ is attained in $<k$ elements $g\in G$. Hence some $l<k$ maximal elements of $\phi(G)$ are separated from the other elements, that is $\phi\in V_l$ for some $l<k$. The inverse reasoning is also valid.

Assertion~(\ref{a_cont}) now follows from~(\ref{a_action}) and~(\ref{a_uni}), because the domain $V_m$ of $f_m$ is a union of disjoint open sets, and $f_m$ is constant on each of these sets by definition.

Now it follows from assertion~(\ref{a_cont}) that $g_G(V_m)=1$, from assertion~(\ref{a_uni}) that
$$
g_G(W(q,m)\setminus W(q,m-1)) \le g_G(V_m) = 1,
$$ 
and by the subadditivity of genus and assertion~(\ref{a_uni}) we obtain the inequality $g_G(W(q, k))\le k-1$ by induction.
\end{proof}

\begin{thm}
\label{conf-space-genus} The genus of the configuration-like spaces, under the action of $G$ described above, satisfies the upper bounds
$$
g_G(V(1,q,k))\le g_G(W(q, k))\le k-1
$$ 
and
\begin{multline*}
g_G(V(m,q,k))\le g_G(V_1(m,q,k))\le \\ \le g_G(V(m-1,q,q))+k-1\le (m-1)(q-1)+k-1.
\end{multline*}
\end{thm}

\begin{proof}
The first bound follows directly from Theorem~\ref{max-conf-space-genus} and the monotonicity of genus.

To prove the second bound, we note that under the decomposition $\mathbb R^m = \mathbb R\times \mathbb R^{m-1}$ we have
$$
V_1(m, q, k) \subseteq \left(V(m-1, q, q)\times \mathbb R^q\right) \times \left((\mathbb R^{m-1})^q\times V(1, q, k)\right).
$$
The right summands are $G$-equivariantly projected to $V(m-1,q,q)$ and $V(1,q,k)$ respectively, thus the subadditivity and the monotonicity of genus give the inequality
$$
g_G(V_1(m,q,k))\le g_G(V(m-1,q,q))+g_G(V(1, q, k)).
$$
Note that $V(m-1, q, q)$ is homotopy equivalent to $(m-1)(q-1)-1$-dimensional sphere with action of $G$ without fixed points, so by the dimension upper bound we have 
$$
g_G(V(m-1,q,q))\le (m-1)(q-1),
$$
and therefore
$$
g_G(V_1(m,q,k))\le (m-1)(q-1)+k-1.
$$
The inclusion $V(m, q, k)\subseteq V_1(m, q, k)$ implies that
$$
g_G(V(m,q,k))\le g_G(V_1(m,q,k)),
$$
which completes the proof.
\end{proof}

Now we are going to give some estimates on the genus of configuration-like spaces from below. First, we consider an arbitrary finite group $G$; in this case the bounds are expressed in terms of the genus $g_G(V(Y, q, q))$. In the case $Y=\mathbb R^m$, the space $V(m, q, q)$ is homotopy equivalent to $I[G]^m\setminus\{0\}$ (see the definition before Corollary~\ref{knaster1}), i.e. a $((q-1)m-1)$-dimensional sphere. 

For the case of an arbitrary finite group $G$, the genus $g_G(V(m, q, q))$ is not known. For $p$-tori the cohomology lower bound and the dimension upper bounds coincide; in this case $g_G(V(m, q, q))=(q-1)m$. Let us formulate lower bounds for $g_G(V(m, q, k))$ for an arbitrary finite group $G$; the case of $p$-tori is considered in the end of this section.

\begin{lem}
\label{lower-bound-onedim}
For the spaces $W(q,k)$ and $V(1,q,k)$ we have 
$$
g_G(W(q,k)) \ge g_G(W(q,q))-q+k = g_G(V(1, q, q))-q+k.
$$
\end{lem}

\begin{proof}
The inequality was actually proved in the proof of Theorem~\ref{max-conf-space-genus}, because (in the notation of that proof)
$$
W(q,q)\subseteq W(q,k)\cup V_k\cup V_{k+1}\cup\dots\cup V_{q-1}
$$
and we obtain the needed inequality by the subadditivity of genus.
\end{proof}

\begin{lem}
\label{lower-bound-klarge}
For an arbitrary metric space $Y$ we have 
$$
g_G(V(Y,q,k)) \ge g_G(V(Y,q,q))-q+k,
$$
when $k > q/2$.
\end{lem}

\begin{proof}
Note that for a metric space $Y$ the spaces $V(Y, q, k)$ are metric, and hence paracompact. Thus the additivity of the genus holds for such configuration-like spaces and their subsets.

Let us argue by descending induction on $k$. It is clear that
$$
V(Y, q, k+1)\setminus V(Y, q, k) = \Delta_q^k(Y)\setminus \Delta_q^{k+1}(Y).
$$
Let us map $X_k = \Delta_q^k(Y)\setminus \Delta_q^{k+1}(Y)$ to $\binom{G}{k}$ by assigning to the configuration 
$$
(x_1,\ldots, x_q)\in X_k
$$ 
the (unique, since $k>q/2$) $k$-element subset $M=f((x_1,\ldots, x_q))\subset [q]$ ($[q]$ is identified with $G$) such that $\{x_i\}_{i\in M}$ is a one point set. The map $f$ is locally constant on $X_k$; let us show this explicitly. For any $k$-element $M\subset [q]$ set
$$
U_M = \{(x_1, \ldots, x_q)\in Y^q : |\{x_i\}_{i\in M}| = 1 \},
$$
a closed set. The sets $U_M\cap X_k$ are disjoint, closed in $X_k$, and cover $X_k$. Since there is a finite number of such sets, then they are also open in $X_k$. Now it remains to note that $f$ is constant on every $U_M\cap X_k$.

The existence of the above map $f$ implies
$$
g_G(V(Y, q, k+1)\setminus V(Y, q, k))\le 1.
$$
Now the induction step is made by the subadditivity of $g_G$:
$$
g_G(V(Y, q, k+1)) \le g_G(V(Y, q, k)) + 1.
$$
\end{proof}

Now we are going to give some exact formulas for the genus in the case of $p$-tori.

\begin{thm} 
\label{precise-genus-p-tori}
If $G$ is a $p$-torus then 
$$
g_G(V(1, q, k)) = g_G(W(q,k))= k - 1,
$$ 
and 
$$
g_G(V_1(m, q, k)) = (m-1)(q-1)+k-1.
$$
If, in addition, $k > q/2$, then
$$
g_G(V(m,q,k))=(m-1)(q-1)+k-1.
$$
\end{thm}

\begin{proof}
The space $V(1, q, k)$ is a complement to a system of $(q-k+1)$-dimensional linear subspaces in $\mathbb R^q$. Thus it is $(k-3)$-connected, where ``$(-1)$-connected'' means ``arbitrary space''. Thus the cohomology lower bound (for a $p$-toral action) and the monotonicity of the genus give
$$
g_G(W(q, k)) \ge g_G(V(1, q, k)) \ge k-1.
$$
From Theorem~\ref{max-conf-space-genus} we have $g_G(W(q, k))\le k-1$, thus all the inequalities are equalities.

Similarly, the space $V_1(m, q, k)$ is a complement in $\mathbb R^{mq}$ to a system of $(m + q - k)$-dimensional linear subspaces, thus it is $c$-connected for 
$$
c = (m-1)(q-1)+k-3,
$$ 
and from the cohomology lower bound we have 
$$
g_G(V_1(m ,q, k)) \ge (m-1)(q-1)+k-1,
$$
which coincides with the upper bound in Theorem~\ref{conf-space-genus}.

Now consider the space $V(m, q, k)$. Note that
$$
g_G(V(m,q,q))=m(q-1),
$$ 
because this configuration space is a homotopy sphere, and its cohomology lower bound for $g_G$ coincides with the dimension upper bound. If $k > q/2$, the lower bound for $g_G(V(m,q,k))$ is obtained from Lemma~\ref{lower-bound-klarge}, and coincides with the upper bound in Theorem~\ref{conf-space-genus}.
\end{proof}

If we consider an arbitrary $m$-dimensional manifold $M$, then the corresponding configuration space $V(M, q, k)$ obviously contains a copy of $V(m, q, k)$. Hence, from monotonicity, we have
$$
g_G(V(M, q, k)) \ge g_G(V(m, q, k)).
$$
Let us improve this bound by $+1$ in one particular case.

\begin{thm}
\label{lower-bound-manifold}
Let $G$ be a $p$-torus, let $M$ be a smooth oriented (if $p\not=2$) closed manifold of dimension $m$, and let $k>q/2$. Then
$$
g_G(V(M, q, k)) \ge (m-1)(q-1) + k.
$$
\end{thm}

\begin{proof}
In~\cite{kar2009} it was shown that
$$
g_G(V(M, q, q))\ge i_G(V(M, q, q)) \ge m(q-1) + 1.
$$
Now Lemma~\ref{lower-bound-klarge} implies the required inequality.
\end{proof}

\begin{rem}
In~\cite{schw1966} it was shown that some additional restrictions on the manifold $M$ may give better lower bounds on $g_G(M,q,q)$ for $q=p$ and $G=\Zp$. Hence, similar to Theorem~\ref{lower-bound-manifold}, the genus $g_G(M,q,k)$ for $k>q/2$ has better bounds from below. The restrictions on $M$ are expressed in terms of Smith's operations in homology (or, equivalently, certain characteristic classes of $M$).
\end{rem}

\begin{rem}
A lower bound of $g_G(V(m, q, k))$ for a $p$-torus $G$ and arbitrary $m$ and $k$ is also obtained in~\cite{kar2009}, estimating the genus by the homological index $i_G$. In the case $G=\Zp$ ($p$ is a prime) that lower bound coincides with the upper bound of Theorem~\ref{conf-space-genus}, while in the case of general $p$-tori the lower bound takes the form:
$$
g_G(V(m,q,k))\ge (m-1)(q-q/p)+k-1.
$$
Thus, in the case $k\le q/2$ and $q$ not a prime, the gap between the lower and the upper bounds still remains.
\end{rem}

In the case $q=p$ we have an even better bound for $k=2$ (the classical configuration space).

\begin{lem}
\label{lower-bound-pconf}
If $G=\Zp$, so $q=p$, then for an arbitrary metric space $Y$ we have
$$
g_G(V(Y, p, 2)) \ge g_G(V(Y,p,p)) - p + 2.
$$
\end{lem}

\begin{proof}
Let us define another configuration space ($|\cdot|$ denotes the cardinality of a finite set)
$$
U(Y, p, l) = \{(x_1,\ldots, x_p)\in Y^p : |\{x_1,\ldots,x_p\}|\ge l \}.
$$
Obviously, $U(Y,p,2) = V(Y,p,p)$ and $U(Y,p,p) = V(Y,p,2)$.

For any $k = 2,\ldots,p-1$, the difference $X_k = U(Y,p,k)\setminus U(Y,p,k+1)$ has genus $1$. This
can be shown as follows (similar to the proof of Lemma~\ref{lower-bound-klarge}). Call a partition $\mathcal P = \{A_1, \ldots, A_s\}$ of the set $[p]$ a \emph{pattern of coincidence}, and denote the subset of the Cartesian power that has these coincidences (and possibly other coincidences) by
$$
F_{\mathcal P} = \{(x_1,\ldots, x_p)\in Y^p : \forall t=1,\ldots,s\ |\{x_i\}_{i\in A_t}| =1\};
$$
these sets are closed. Note that the sets $X_k\cap F_{\mathcal P}$ for $|\mathcal P| = k$ give a partition of $X_k$, hence each $X_k\cap F_{\mathcal P}$ is open in $X_k$. Now we assign to every $(x_1,\ldots, x_p)\in X_k$ its unique $k$-element pattern of coincidence, thus obtaining a locally constant map $f$ from $X_k$ to the set of all patterns of size $k$. This map is equivariant, and $\Zp$ acts on such patterns without fixed points if $k\in [2,p-1]$ (here we essentially use $G=\Zp$).

The sets $U(Y,p,k)$ are closed, hence by the subadditivity of the genus
\begin{multline*}
g_G(U(Y,p,p)) \ge g_G(U(Y,p,p-1)) - 1 \ge \\ \ge \dots \ge g_G(U(Y,p,2)) - p + 2.
\end{multline*}

\end{proof}

Using the inequality from~\cite{kar2009} for $\Fp$-oriented manifolds
$$
g_G(V(M,p,p)) \ge i_G(V(M,p,p)) \ge m(p-1) + 1,
$$
we obtain the following lower bound on the genus of the classical configuration space of a closed manifold.

\begin{cor}
\label{lower-bound-pconf-manifold}
Let $G=\Zp$, so $q=p$, and let $M$ be a smooth oriented (if $p\neq 2$) closed manifold of dimension $m$. Then
$$
g_G(V(M, p, 2)) \ge (m-1)(p-1) + 2.
$$
\end{cor}

\begin{rem}
This corollary gives a shorter proof (without spectral sequences) of the main theorem in~\cite{kar2009bil}.
\end{rem}

\section{Proof of the coincidence theorem}
\label{coinc-proof-sec}

First we state the main tool to estimate the genus of inverse images under $G$-maps.

\begin{thm}
\label{inverse-image-genus}
Let $X$ be a fixed point free $G$-space, and let $E$ be a $G$-space with fixed point set contained in a closed $G$-invariant subspace $P\subset E$. Let $f : X\to E$ be an equivariant map. Then $g_G(f^{-1}(P))\ge g_G(X)-g_G(E\setminus P)$. In particular $f^{-1}(P)\not=\emptyset$ if $g_G(X)>g_G(E\setminus P)$.
\end{thm}

\begin{proof} We have $X=f^{-1}(P)\cup f^{-1}(E\setminus P)$, and since $g_G(f^{-1}(E\setminus P))\le g_G(E\setminus P))$ by the monotonicity of genus, we obtain the desired inequality from the subadditivity of genus.
\end{proof}

Applying this theorem several times we obtain the following statement.

\begin{thm}
\label{inverse-image-intersection-genus}
Let $f_i:X\to E_i$, $i=1,\dots,r$ be $G$-maps and let closed $G$-invariant subspaces $P_i\subset E_i$ contain all fixed points of their respective $E_i$. Then
$$
g_G(\bigcap_{i=1}^r f_i^{-1}(P_i))\ge g_G(X)-\sum_{i=1}^r g_G(E_i\setminus P_i).
$$
\end{thm}

Now let us give some definitions.

\begin{defn}
Let $X$ be a $G$-space and let $f : X\to Y$ be a continuous map. Define the map $\widehat f : X\to \Map(G, Y)$ by the formula
$$
\widehat f(x)(g) = f(g^{-1}x).
$$
Note that for any $h\in G$
$$
h(\widehat f(x))(g) = \widehat f(x)(h^{-1}g) = f((h^{-1}g)^{-1}x) = f(g^{-1}hx) = \widehat f(hx)(g),
$$
so $h(\widehat f(x)) = \widehat f(hx)$, thus the map $\widehat f$ is $G$-equivariant.
\end{defn}

\begin{lem}
\label{coincidence-redifinition}
The orbit coincidence sets can be defined as follows $(q=|G|)$:
$$
A(f, k) = \widehat f^{-1}(\Delta_q^k(Y))
$$
and
$$
A'(f, k) = \widehat f^{-1}(\Deltad_q^k(\mathbb R)).
$$
\end{lem}

Now we are ready to prove the main result.

\begin{proof}[Proof of Theorem~\ref{coinc-fpf}]
Put $E_1=\mathbb R^q$, $P_1=\Deltad_q^k(\mathbb R)$, $E_2=(\mathbb R^{m-1})^{q}$, $P_2=\Delta_q^q(\mathbb R^{m-1})$. From Theorems~\ref{max-conf-space-genus}, \ref{conf-space-genus}, \ref{inverse-image-intersection-genus}, and Lemma~\ref{coincidence-redifinition} we obtain 
\begin{multline*}
g_G(A'(f_1, k)\cap A(h, q))\ge\\
\ge g_G(X)-g_G(W(q,k))-g_G(V(m-1,q,q))\ge\\ \ge g_G(X)-(m-1)(q-1)-k+1.
\end{multline*}

Consider $f : X\to \mathbb R^m$ as a pair $f=f_1\oplus h$. It is clear that $A(f, k)\supseteq A'(f_1, k)\cap A(h, q)$.
\end{proof}

In fact, the above reasoning prove the following statement.

\begin{thm}
Let $G$ be a finite group, $m\ge 0$. Consider a fixed point free $G$-space $X$, functions $f_i:X\to \mathbb R$ $(i=1,\ldots, d)$, integers $2\le k_i\le q$ $(i=1,\ldots, d)$, and a map $h:X\to \mathbb R^m$. Then
\begin{multline*}
g_G(A(h, q))\cap\bigcap_{i=1}^d A'(f_i, k_i)) \ge \\
\ge g_G(X)-\sum_{i=1}^d g_G(W(q,k_i))-g_G(V(m,q,q))\ge \\
\ge g_G(X)-m(q-1)-\sum_{i=1}^d k_i+d.
\end{multline*}
\end{thm}

\section{Coincidences of maps from manifolds}
\label{manifolds}

In this section we prove a coincidence theorem for the case in which $X$ is a manifold. In this case the simple estimate $\dim A(f, k) \ge g_G(A(f, k)) - 1$ can be improved by imposing additional assumptions.

\begin{thm}
Let $X$ be an $\Fp$-orientable compact connected $N$-di\-men\-sion\-al manifold and assume $H^l(X, \Fp) = 0$ for all positive $l < (m-1)(q-1)+k-1$. Suppose that a $p$-torus $G$ $(|G|=q)$ acts on $X$ without fixed points and $f: X\to \mathbb R^m$ is a continuous map. Then 
$$
\dim A(f,k)\ge N-(m-1)(q-1)-k+1.
$$
\end{thm}

\begin{proof}
Observe that $\widehat f$ restricts to an equivariant map of $X\setminus A(f, k)$ into $V(m, q, k)$ and that we may assume that $X \setminus  A(f, k)$ is path connected. 

Then $g_G(X \setminus  A(f, k))\le g_G(V(m,q, k))$ from the monotonicity of $g_G$. From the cohomology lower bound for genus, there must be some $j$ , $0 < j < g_G(V(m, q, k))$, such that $H^j(X\setminus A(f,k), \Fp)\not=0$, and hence $H_j(X\setminus A(f, k), \Fp) \not=0$. By Poincar\'e duality we obtain 
$$
H^{N-j}(X, A(f,k), \Fp) \not=0.
$$

From the statement of the theorem $H_j(X) = 0$, and by Poincar\'e duality $H^{N-j}(X) = 0$. Thus by the exact cohomology sequence of the pair $(X, A(f, k))$ we obtain 
$$
H^{N-j-1}(A(f, k)) \not=0.
$$ 
It follows that
$$
\dim A(f, k) \ge N - j - 1 \ge N - g_G(V(m,q, k))\ge N-(m-1)(q-1)-k+1.
$$
\end{proof}

\end{document}